\documentclass[reqno]{amsart}
\usepackage{amssymb,amscd,enumerate,mathrsfs,curves}

\numberwithin{equation}{section}
\newtheorem{theorem}[equation]{Theorem}

\newtheorem{lemma}[equation]{Lemma}

\theoremstyle{definition}
\newtheorem{definition}[equation]{Definition}
\newtheorem{remark}[equation]{Remark}

\def\p{\mathfrak p}
\def\PP{\mathfrak P}
\def\QQ{\mathfrak Q}
\def\q{\mathfrak q}
\def\bP{\,{}^b\!P}
\def\bA{\,{}^b\!A}
\def\C{\mathbb C}
\def\NN{\mathbb N}
\def\R{\mathbb R}
\def\B{\mathcal B}
\def\Crack{\mathcal C}
\def\Dom{\mathcal D}
\def\M{\mathcal M}
\def\N{\mathcal N}
\def\trb{\mathscr T}
\def\P{\mathcal P}
\def\kerb{\mathscr K}
\def\Y{\mathcal Y}
\def\Z{\mathcal Z}
\def\BB{\mathscr B}
\def\scrD{\mathscr D}
\def\Ha{\mathscr H}
\def\scrP{\mathscr P}
\def\Hol{\mathfrak{H}}
\def\Mero{\mathfrak{M}}
\def\pole{\mathfrak{pole}}
\def\ss{\mathfrak s}
\def\eps{\varepsilon}
\def\minus{\backslash}
\def\open#1{\smash[t]{\overset{{}_{\,\,\circ}}{#1}{}}}
\def\set#1{\{#1\}}
\def\im{i}
\def\embed{\hookrightarrow}
\def\display#1#2{\mbox{\parbox{#1} {#2}}}

\DeclareMathOperator{\Diff}{Diff}
\DeclareMathOperator{\End}{End}
\DeclareMathOperator{\spec}{spec}
\DeclareMathOperator{\sing}{sing}
\DeclareMathOperator{\tr}{tr}
\DeclareMathOperator{\dist}{dist}

\begin{document}

\title{The kernel bundle of a holomorphic Fredholm family}
\thanks{Work partially supported by the National Science Foundation, Grants DMS-0901202 and DMS-0901173}

\author{Thomas Krainer}
\address{Penn State Altoona\\ 3000 Ivyside Park \\ Altoona, PA 16601-3760}
\email{krainer@psu.edu}
\author{Gerardo A. Mendoza}
\address{Department of Mathematics\\ Temple University\\ Philadelphia, PA 19122}
\email{gmendoza@temple.edu}

\begin{abstract}
Let $\Y$ be a smooth connected manifold, $\Sigma\subset\C$ an open set and $(\sigma,y)\to\scrP_y(\sigma)$ a family of unbounded Fredholm operators $D\subset H_1\to H_2$ of index $0$ depending smoothly on $(y,\sigma)\in \Y\times \Sigma$ and holomorphically on $\sigma$. We show how to associate to $\scrP$, under mild hypotheses, a smooth vector bundle $\kerb\to\Y$ whose fiber over a given $y\in \Y$ consists of classes, modulo holomorphic elements, of meromorphic elements $\phi$ with $\scrP_y\phi$ holomorphic. As applications we give two examples relevant in the general theory of boundary value problems for elliptic wedge operators.

\end{abstract}


\maketitle


\section{Introduction}

It has long been recognized that, as with uniformly elliptic linear operators on smooth bounded domains, also for other classes of linear elliptic operators $A$ it is the case that boundary conditions should be expressed as conditions on (some of) the coefficients of the asymptotic expansion at the boundary of formal solutions of $Au=f$. Such asymptotic expansions are proved to exist in many instances, for example for  elliptic $b$-operators or operators of Fuchs type (Kondrat$'$ev \cite{Kond67}, Melrose \cite{Mel93}, Melrose and Mendoza \cite{MM81,MM83}, Rempel and Schulze \cite{RemSch1985}, Schulze \cite{Schulze94}, in an analytic context Igari \cite{Igari}, etc.), where such expansions, well understood, form a finite dimensional space (Lesch \cite{Lesch1997}). More generally, such expansions also exist under some conditions for elliptic $e$-operators (Mazzeo \cite{Maz91}, and in a somewhat different context, Costabel and Dauge \cite{CostabelDauge90,CostabelDauge93} and Schmutzler \cite{Schm92,Schm97}). These classes, which include regular elliptic differential operators, come up in certain geometric problems on noncompact manifolds.  Slightly modified, they are central in the analysis of elliptic problems on compact manifolds with singularities (conical points, edges, and corners, for instance). 

Because of their importance and potential applicability there is great interest in developing the tools to handle boundary value problems for these classes. To this end we present here the definition of a smooth vector bundle whose sections are the principal parts of the analogues of the traces (in the sense of boundary values) in the case of classical problems. Boundary conditions are to be imposed as differential or pseudodifferential conditions on these sections. We refer to the vector bundle as the trace bundle. It depends, in general, on the differential operator itself.  

To give some concrete context, consider first the case where $A$ is a regular elliptic linear differential operator of order $m>0$ on a manifold with boundary. After localization and flattening of the boundary, 
\begin{equation*}
A=\sum_{k+|\alpha|\leq m} a_{k,\alpha}(x,y)D_x^kD_y^\alpha
\end{equation*}
in a neighborhood of $0$ in $\overline \R_+\times\R^n$. The coefficients are smooth, the boundary is $x=0$ and the interior of the manifold contains the region $x>0$. The operator $P=x^mA$ is an example of an edge operator, and $x^{-m}P$ (so $A$ itself) an example of a wedge operator. The operator $P$ is equal to 
\begin{equation*}
\sum_{k+|\alpha|\leq m} a_{k,\alpha}(x,y)x^{m-k-|\alpha|}p_k(xD_x+\im|\alpha|) (xD_y)^\alpha
\end{equation*}
where $p_k(\sigma)=(\sigma+\im(k-1))(\sigma+\im(k-2))\cdots\sigma$. The Ansatz that a formal solution of $Pu=0$ is a formal series in (possibly complex) powers of $x$ at $x=0$ leads to the conclusion that in fact $u$ has a classical Taylor expansion,
\begin{equation*}
u\sim \sum_{k=0}^\infty u_k(y)x^k.
\end{equation*}
The primary reason for this can be seen through the same analysis as with ordinary differential operators with regular singular points: the indicial equation for the operator above is $a_{m,0}(0,y)p_m(\sigma)=0$, an equation whose roots are the elements of 
\begin{equation*}
\mathcal I=\set{0,-\im,-2\im,\dots, -\im (m-1)}.
\end{equation*}
These are the only roots because $a_{m,0}(0,y)$ is invertible due to the assumed ellipticity of $A$. The coefficients of $x^{\im \sigma}$ with $\sigma\in \mathcal I$ in the formal Taylor expansion of a solution of $Au=0$ are the objects on which conditions are placed. The upshot of this brief analysis is that if $A$ acts on sections of a vector bundle $E$, then the trace bundle of $A$ is the direct sum of $m$ copies of the restriction of $E$ to the boundary; this may be viewed simply as the kernel of $x^m D_x^m$. This leads to the trivial complex vector bundle of rank $m$ if $A$ is a scalar operator.

This analysis generalizes, with a different conclusion, to operators of the form $A=x^{-m}P$ where $P$ is an edge operator,
\begin{equation*}
P=\sum_{k+|\alpha|+|\beta|\leq m} a_{k,\alpha,\beta}(x,y,z)(xD_x)^k (xD_y)^\alpha D_z^\beta,
\end{equation*}
(see \cite{Maz91}) where again $a_{k,\alpha,\beta}$ is smooth and, in the simplest case, $z$ ranges in a compact manifold $\Z$ without boundary; $\Z$ is a point in the case of a regular elliptic operator. Loosely speaking, the variable $y$ lies in an open set some Euclidean space, more generally in manifold $\Y$, while $x\geq 0$. The case $\Y=\set{\mathrm{pt.}}$ models conical singularities. Edge-ellipticity of $P$ is the property that its edge symbol,
\begin{equation*}
\sum_{k+|\alpha|+|\beta|= m} a_{k,\alpha,\beta}(x,y,z)\xi^k \eta^\alpha \zeta^\beta,
\end{equation*}
is invertible if $(\xi,\eta,\zeta)\ne 0$. If $P$ is $e$-elliptic, its indicial family,
\begin{equation*}
\P_y(\sigma)=\sum_{k+|\alpha|\leq m} a_{k,0,\beta}(0,y,z)\sigma^k D_z^\beta,
\end{equation*}
is elliptic as a family of differential operators on $\Z$ (typically acting on sections of a vector bundle). This family depends smoothly on $(y,\sigma)$ and holomorphically on $\sigma\in \C$. For any given $y$ the elements $\sigma$ for which $\P_y(\sigma)$ is not invertible forms a closed discrete set $\spec_b(P_y)\subset \C$ called the boundary spectrum of $A$ at $y$ (see \cite{Mel93}); by a well known trick, this set \emph{is} the spectrum of an operator. Quite evidently, $\spec_b(P_y)$ typically depends on $y$. If $\Z$ is a manifold with boundary, then the indicial family is accompanied by homogeneous boundary conditions coming from the original setup---again making the $b$-spectrum discrete for each $y$. The possible dependence on $y$ may lead, because of variable multiplicity, to the effects referred to in the literature as branching asymptotics making the analysis of boundary value problems for $A=x^{-m}P$ considerably more difficult in comparison with the classical case.

It is this branching that we address here. What concerns us is the global definition of the trace bundle, specifically, its $C^\infty$ structure: in order to develop a general theory of boundary value problems for elliptic operators such as $A$, a global problem by its very nature (consider for instance the APS boundary condition), one needs to be able to refer to traces as global objects.

\medskip
The rest of the paper is organized as follows. 

We collect in Section~\ref{sec-SetUp} all the assumptions underlying our construction of a vector bundle and definition of its $C^\infty$ structure, of a  vector bundle closely related to the trace bundle. This vector bundle, which we call the kernel bundle, is defined, as a set, in Section~\ref{se-kerb}; see Theorem~\ref{TheKernelBundle}. Our approach---defining the kernel bundle rather than a trace bundle directly---allows for enough generality to treat at the same time trace bundles both when $\Z$ closed and when $\partial\Z\ne \emptyset$ (the latter case with some boundary condition). Allowing for such generality opens the door to use an iterative approach to handle the more complicated situations that arise in the presence of a stratification of the boundary.

We construct the putative smooth frames in Section~\ref{sec-Frames}. An equivalent version of these local frames was already defined by Schmutzler \cite{Schm92,Schm97} using Keldysh chains, however, these papers do not address the regularity of the transition functions.

Our construction of the special frames lends itself to be readily used to show, in Section~\ref{sec-TransitionFunctions}, that two frames of the same kind are related by smooth transition functions, in other words, that the kernel bundle has a smooth vector bundle structure, see Theorem~\ref{TheFrame}. This structure is natural in that the property that a section is smooth can be checked intrinsically. The proof of Theorem~\ref{TheFrame} relies heavily on the smoothness and nondegeneracy of a pairing between the kernel bundle (associated to a family of operators) and that of its dual. This result is stated as Theorem~\ref{NonDegPairing}; its proof relies heavily on ideas used in \cite{GiMe01}. Section~\ref{sec-TransitionFunctions} is at the heart of this work. 

The last two sections provide examples. In Section~\ref{sec-TraceWedge}, the first of these two sections, we consider the case of a general elliptic wedge operator $x^{-m}P$, $P\in \Diff^m_e$, where the boundary fibration has compact fibers. In the second, Section~\ref{sec-Example}, we illustrate with a toy example (motivated by what would be codimension $1$ cracks in linear elasticity) the use of the kernel bundle when the fibers of the boundary fibration (the $\Z$) are nonclosed.

\section{Set up}\label{sec-SetUp}

Let $\Y$ be a connected manifold and
\begin{equation*}
\wp_1:\Ha_1\to\Y, \qquad \wp_2:\Ha_2\to\Y
\end{equation*}
smooth Hilbert space bundles. Further let $\scrD\to\Y$ be another smooth Hilbert space bundle continuously embedded in $\Ha_1$ with fiberwise dense image and such that the trivializations of $\Ha_1$ (smooth and unitary) restrict to smooth trivializations of $\scrD$. We write $H_1$, $H_2$ and $\Dom$ for the model spaces. 

Let $\Sigma$ an open connected subset of $\C$. With $\pi:\Y\times\Sigma\to \Y$ denoting the canonical projection, let 
\begin{equation}\label{DataInBundleContext}
\scrP:\pi^*\scrD\subset \pi^*\Ha_1\to \pi^*\Ha_2
\end{equation}
be a smooth bundle homomorphism covering the identity consisting of fiberwise closed Fredholm operators depending holomorphically on $\sigma$. Suppose further that for each $y\in \Y$ there is $\sigma\in \Sigma$ such that
\begin{equation}\label{PointwiseOps}
\scrP_y(\sigma):\scrD_y\subset \Ha_{1,y}\to \Ha_{2,y}
\end{equation}
is invertible.
Passing to trivializations over an open set $U\subset \Y$, \eqref{DataInBundleContext} becomes a smooth family
\begin{equation*}
\P_y(\sigma):\Dom\subset H_1\to H_2,\quad (y,\sigma)\in U\times\Sigma,
\end{equation*}
holomorphic in $\sigma$.

Since \eqref{PointwiseOps} is Fredholm for all $\sigma\in \Sigma$ and invertible for some such $\sigma$, 
\begin{equation*}
\sing_b(\scrP_y)=\set{\sigma\in \Sigma:\scrP_y(\sigma)\text{ is not invertible}}
\end{equation*}
is a closed discrete subset of $\Sigma$ and 
\begin{equation*}
\sing_e(\scrP)=\set{(y,\sigma)\in \Y\times\Sigma:\sigma\in \sing_b(\scrP_y)}.
\end{equation*}
is a closed subset of $\Y\times \Sigma$. The notation is motivated by the corresponding objects in the context of $b$- and $e$-operators, $\spec_b$ and $\spec_e$ (see \cite{Mel93}, \cite{Maz91}, also \cite{GKM4}). We will assume the stronger condition that 
\begin{equation}\label{FiniteSpecb}
\display{310pt}{$\sing_b(\scrP_y)$ is a finite set for each $y\in \Y$ and $\sing_e(\scrP)$ is closed in $\Y\times \C$.}
\end{equation}
The condition that $\sing_e(\scrP)$ is closed in $\Y\times \Sigma$ means that for every $y_0\in \Y$ there is $\delta>0$ and a neighborhood $U$ of $y_0$ such that $\dist(\sing_b(\scrP_y),\C\minus\Sigma)>\delta$ if $y\in U$.

Write $\scrP_y(\sigma)^*$ for the Hilbert space adjoint of \eqref{PointwiseOps}. Assume that the domains of the $\scrP_y(\sigma)^*$ join to give another smooth Hilbert space bundle $\scrD^*$ continuously embedded in $\Ha_2$ and such that the smooth unitary trivializations of $\Ha_2$ restrict to trivializations of $\scrD^*$. Defining $\scrP^*_y(\sigma)=\scrP_y(\overline \sigma)^*$ we obtain another smooth homomorphism
\begin{equation*}
\scrP^*:\pi^* \scrD^*\subset \Ha_2\to \pi^*\Ha_1
\end{equation*}
depending holomorphically on $\sigma$, where now $\pi$ is the projection $\Y\times\overline \Sigma\to \Y$, which satisfies the same Fredholm, analyticity, and invertibility properties as $\scrP$.

\section{The kernel bundle}\label{se-kerb}
If $K$ is a Hilbert space and $V\subset \C$ is open, we write $\Mero(V,K)$ for the space of meromorphic $K$-valued functions on $V$ and $\Hol(V,K)$ for the subspace of holomorphic elements. Thus $f\in \Mero(V,K)$ if there is, for each $\sigma_0\in V$, a number $\mu_0\in \NN_0$ such that $\sigma\mapsto (\sigma-\sigma_0)^{\mu_0}f(\sigma)$ is holomorphic near $\sigma_0$. Suppose $\Omega\Subset V$ is open and has smooth (or rectifiable) boundary. If $f\in\Mero(V,K)$ has finitely many poles in $\Omega$ and no poles on $\partial\Omega$, then the sum of the singular parts of $f$ at each pole in $\Omega$ is given by 
\begin{equation*}
\ss_V(f)(\sigma)=\frac{\im}{2\pi}\oint_{\partial\Omega}\frac{f(\zeta)}{\zeta-\sigma}\,d\zeta,\quad \ |\sigma|\gg 1
\end{equation*}
with the positive orientation for $\partial\Omega$. Replacing $\Omega$ by a disjoint union of open discs with small radii, each containing at most a single pole of $f$ and contained in $\Omega$ we see that the formula determines an element of $\Mero(\C,K)$. 

\medskip
Let $V$ be an open subset of $\Sigma$. Since $\sigma\mapsto \scrP_y(\sigma)$ is holomorphic, it gives maps
\begin{equation*}
\scrP_y:\Mero(V,\scrD_y)\to \Mero(V,\Ha_{2,y}),\quad \scrP_y:\Hol(V,\scrD_y)\to \Hol(V,\Ha_{2,y})
\end{equation*}
so there is an induced map
\begin{equation}\label{AsKernel}
[\scrP_y]_V:\Mero(V,\scrD_y)/\Hol(V,\scrD_y)\to \Mero(V,\Ha_{2,y})/\Hol(V,\Ha_{2,y}).
\end{equation}
Any element $[\phi]\in \ker [\scrP_y]_V$ is represented uniquely by the sum of the singular parts of any given representative $\phi$ at the various poles in $V$. Define
\begin{equation*}
\kerb_y=\set{\ss_\Sigma(\phi):\phi\in \Mero(\Sigma,\scrD_y),\ \scrP_y\phi\in \Hol(\Sigma,\Ha_{2,y}) }.
\end{equation*}
Thus $\kerb_y$ is canonically isomorphic to the kernel of $[\scrP_y]_\Sigma$. It is a vector space over $\C$, finite-dimensional due to \eqref{FiniteSpecb} and the various other hypothesis made on $\scrP$. An element of $\kerb_y$ is in particular a $\scrD_y$-valued meromorphic function on $\C$ with poles contained in $\sing_b(\scrP_y)$; and if the element is regular, then it is the zero function. 

It is also convenient to define, if $\sigma_0\in \Sigma$
\begin{equation*}
\kerb_{y,\sigma_0}=\set{\ss_D(\phi):\phi\in \Mero(D,\scrD_y):\scrP_y\phi\in \Hol(D,\Ha_{2,y})}
\end{equation*}
where $D$ is a disc in $\Sigma$ centered at $\sigma_0$ with
\begin{equation*}
\sing_b(\scrP_y)\cap D\minus\set {\sigma_0}=\emptyset.
\end{equation*}
Thus
\begin{equation*}
\kerb_y=\bigoplus_{\sigma_0\in \Sigma}\kerb_{y,\sigma_0}.
\end{equation*}

\begin{theorem}\label{TheKernelBundle}
Define
\begin{equation*}
\kerb=\bigsqcup_{y\in \Y} \kerb_y,\quad \pi:\kerb\to \Y\text{ the canonical map},
\end{equation*}
and let $\BB^\infty(\Y;\kerb)$ be the space of right inverses of $\pi$ which viewed as $\Dom$-valued functions on $(\Y\times \C)\minus\sing_e(\P)$ by way of the trivializations of $\Ha_1$ are smooth in the complement of $\sing_e(\scrP)$ and holomorphic in $\sigma$. If \eqref{FiniteSpecb} holds, then $\pi:\kerb\to \Y$ has a smooth vector bundle structure with respect to which its space of $C^\infty$ sections is $\BB^\infty(\Y;\kerb)$.
\end{theorem}

We may replace $\Y$ by an open subset $U\subset \Y$ in all of the above, in which case we naturally write, with only slight abuse of the notation, $\BB^\infty(U;\kerb)$.

\begin{definition}
The vector bundle $\kerb\to\Y$ is the (meromorphic) kernel bundle of $\scrP$.
\end{definition}

Since $\scrP(\sigma)$ commutes with multiplication by functions that depend only on $y$, $\BB^\infty(\Y;\kerb)$ is certainly a module over $C^\infty(\Y)$. 

The proof of Theorem~\ref{TheKernelBundle} will occupy the next two sections. In the first of these, Section~\ref{sec-Frames}, we will find for every $y_0\in \Y$, a neighborhood $U$ of $y_0$ over which $\scrD$, $\Ha_1$, and $\Ha_2$ are trivial and elements 
\begin{equation}\label{LocalFrame}
\phi^s_{j,\ell}\in \BB^\infty(U;\kerb),\quad s=1,\dotsc,S,\ j=1,\dotsc,J_s,\ \ell=0,\dotsc,L_{s,j}-1
\end{equation}
giving a pointwise basis of $\kerb_y$ for each $y\in U$. Then we will show, in Section~\ref{sec-TransitionFunctions}, that for any $\phi\in \BB^\infty(U;\kerb)$, the functions $f^{j,\ell}_s$ such that
\begin{equation*}
\phi=\sum_{s,j,\ell} f^{j,\ell}_s \phi^s_{j,\ell}
\end{equation*}
are smooth. Consequently, declaring the $\phi_{s,j,\ell}$ to be a frame over $U$ gives the desired smooth vector bundle structure. The somewhat peculiar indexing of the components of the frame reflects the nature of the problem, as will be clear as we develop the proof.

\section{Frames}\label{sec-Frames}

Fix $y_0\in \Y$. In the rest of this section we will work in a neighborhood of $y_0$ over which $\scrD$, $\Ha_1$ and $\Ha_2$ are trivial. We write $\P_y(\sigma):\Dom\subset H_1\to H_2$ for the family viewed on the trivialization and write $\sing_b(\P_y)$ in place of $\sing_b(\scrP_y)$ and $\sing_e(\P)$ for the part of $\sing_e(\scrP)$ over the given neighborhood of $y_0$. In this section we aim at finding a neighborhood $U$ of $y_0$ and elements \eqref{LocalFrame} giving a basis of $\kerb_y$ for each $y\in U$.

Denote by $\sigma_s$, $s=1,\dotsc,S_{y_0}$, the points in $\sing_b(\scrP_{y_0})$. Let $K_s\subset \Dom$ and $R_s\subset H_2$ be, respectively, the kernel and image of $\P_{y_0}(\sigma_s)$. The latter space is a closed subspace of $H_2$. The spaces $K_s$ and $R_s^\perp$ are finite-dimensional subspaces of $H_1$ and $H_2$, respectively, of the same dimension because $ \P_{y_0}(\sigma)$ has index $0$. Write $\P_y(\sigma)$ in the form
\begin{equation}\label{FrakA}
\begin{bmatrix}
\p^1_{s,1} & \p^1_{s,2}\\
\p^2_{s,1} & \p^2_{s,2}
\end{bmatrix}:
\begin{matrix}K_s\\ \oplus\\ K_s^\perp\end{matrix}\to \begin{matrix}R_s^\perp\\ \oplus\\ R_s\end{matrix};
\end{equation}
the $\p^i_{s,j}$ are smooth in $(y,\sigma)$ and holomorphic in $\sigma$. The space $K_s^\perp$ is the subspace of $\Dom$ orthogonal to $K_s$ with respect to the inner product of $H_1$. Likewise, the space $R_s^\perp$ is the orthogonal of $R_s$ in $H_2$. All entries except $\p^2_{s,2}$ vanish at $(y_0,\sigma_s)$, and $\p^2_{s,2}(y_0,\sigma_s)$ is invertible. 

There is $\eps>0$ such that
\begin{enumerate}
\item [(1)] \label{Condition1} the family of discs
\begin{equation*}
D_{s,2\eps}=\set{\sigma:|\sigma-\sigma_s| < 2\eps},\quad s=1,\dotsc, S_{y_0}
\end{equation*}
is pairwise disjoint with each $D(\sigma_s,2\eps)$ contained in $\Sigma$; 
\item [(2)]\label{Condition2} for each $s$, $\p^2_{s,2}(y_0,\sigma)$ is invertible if $\sigma\in D_{s,2\eps}$.
\end{enumerate}
By continuity, ($2$) implies there is a connected neighborhood $U$ of $y_0$ such that 
\begin{enumerate}
\item [(3)] \label{Condition3} for each $s$, $\p^2_{s,2}(y,\sigma)$ is invertible if $(y,\sigma)\in U\times D_{s,\eps}$.
\end{enumerate}
Define $\PP_s(y,\sigma) :K_s\to R_s^\perp$ for $(y,\sigma)\in U\times D_{s,\eps}$ by
\begin{equation}\label{FiniteP}
\PP_s=\p^1_{s,1}- \p^1_{s,2}(\p^2_{s,2})^{-1}\p^2_{s,1}.
\end{equation}
In view of ($3$), the invertibility of $\P_y(\sigma)$ if $\sigma\in D_{s,\eps}$ is equivalent to that of $\PP_s(y,\sigma)$. In particular, by ($1$), $\PP_s(y_0,\sigma)$ is invertible if $\sigma\in D_{s,\eps}\minus \set{\sigma_s}$.

The spaces $K_s$ and $R_s^\perp$ have the same dimension so, after fixing a basis for each of these spaces, it makes sense to talk about the determinant of a map $K_s\to R_s^\perp$. Let then $q_s=\det\PP_s$ be computed with respect to some such pair of bases. This is a smooth function on $U\times D_{\sigma_s,\eps}$, holomorphic in $\sigma\in D_{s,\eps}$ and $q_s(y_0,\sigma)$ is nonzero on $D_{s,\eps}\minus \set{\sigma_s}$ by ($1$). Note that
\begin{equation*}
\sing_b(\P_y)\cap D_{s,\eps}=\set{\sigma\in D_{s,\eps}:q_s(y,\sigma)=0}.
\end{equation*}
Shrinking $U$ further we may thus also assume that
\begin{enumerate}
\item [(4)] \label{Condition4} for each $s$, $q_s$ is nowhere zero in $U\times \set{\sigma:\eps/2\leq |\sigma-\sigma_s|<\eps}$, equivalently,
\begin{equation*}
y\in U\implies \sing_b(\P_y)\subset \bigcup_{s=1}^{S_{y_0}} D_{s,\eps/2}.
\end{equation*}
\end{enumerate}
It follows from (4) that the number of zeros, counting multiplicity, of the function
\begin{equation*}
D_{s,\eps}\ni \sigma\mapsto q_s(y,\sigma)\in\C
\end{equation*}
is independent of $y\in U$; we assume throughout that $U$ is connected. Letting $d_s$ be that number, we have in particular that the function $q_s(y_0,\sigma)$ factors as $(\sigma-\sigma_s)^{d_s}h_s(\sigma)$ where $h_s$, defined in $D_{s,\eps}$, is holomorphic and vanishes nowhere in its domain. 

For each $y\in U$ let
\begin{equation*}
\mathfrak K_{s,y} = \set{\ss_{D_{s,\eps}}(\phi):\phi\in \Mero(D_{s,\eps},K_s),\ \PP_s(y,\cdot)\phi\in \Hol(D_{s,\eps},R_s^\perp)}.
\end{equation*}
This space is canonically isomorphic to the kernel of the operator
\begin{equation}
\Mero(D_{s,\eps},K_s)/\Hol(D_{s,\eps},K_s)\to \Mero (D_{s,\eps},R_s^\perp)/\Hol(D_{s,\eps},R_s^\perp)
\end{equation}
induced by $\PP_s(y,\cdot)$. The elements of $\mathfrak K_{s,y}$ are $K_s$-valued meromorphic functions on $\C$ with poles in $D_{s,\eps/2}\cap \sing_b(\P_y)$. Indeed, $\PP_s(y,\sigma)^{-1}$ is meromorphic in $D_{s,\eps}$ with poles in $D_{s,\eps/2}\cap \sing_b(\P_y)$.

By \cite[Lemmas 5.2 and 5.5]{GiMe01} there are elements $\phi^{K_s}_{j,0}(y_0,\sigma)\in \mathfrak K_{s,y_0}$, $j=1,\dotsc,J_s$, with pole only at $\sigma_s$ of some order $L_{s,j}$ such that
\begin{enumerate}
\item [(a)] the elements 
\begin{equation}\label{aBasis}
\phi^{K_s}_{j,\ell}(y_0,\sigma) = \ss_{D_{s,\eps}}\big((\sigma-\sigma_s)^\ell \phi^{K_s}_{j,0}(y_0,\sigma)\big),\quad j=1,\dotsc,J_s,\ \ell=0,\dotsc,L_{s,j}-1
\end{equation}
form a basis of $\mathfrak K_{s,y_0}$,
\item [(b)]the values of the $(\sigma-\sigma_s)^{L_{s,j}}\phi^{K_s}_{j,0}(y_0,\sigma)$ at $\sigma_s$ from a basis of $K_s$,
\item [(c)]the $R_s^\perp$-valued functions
\begin{equation}\label{BasisOfRPerp}
\beta_{s,j}(\sigma) = \PP_s(y_0,\sigma)(\phi^{K_s}_{j,0}(y_0,\sigma))
\end{equation}
are holomorphic in $D_{s,\eps}$ and their values at $\sigma_s$ form a basis of $R_{s}^\perp$.
\end{enumerate}

Define
\begin{equation*}
\phi^{K_s}_{j,\ell}=\ss_{D_{s,\eps}}\big((\sigma-\sigma_s)^\ell \PP_s(y,\sigma)^{-1}\big(\beta_{s,j}(\sigma)\big)\big),\ \ell=0,\dotsc,L_{s,j}-1,\ j=1,\dotsc,J_s.
\end{equation*}

\begin{remark}\label{SmoothnessOfFrame}
Condition (4) above implies that the elements $\phi^{K_s}_{j,\ell}(y,\sigma)$ are holomorphic in the complement of 
\begin{equation*}
(U\times D_{s,\eps/2})\cap \sing_e(A)
\end{equation*}
in $U\times\C$.
\end{remark}

\begin{lemma}\label{PointwiseBasisLemma}
There is a neighborhood $U'\subset U$ of $y_0$ such that for each $y'\in U'$ the restrictions to $\set{y=y'}$ of the functions $\phi^{K_s}_{j,\ell}$ give a basis of $\mathfrak K_{s,y'}$.
\end{lemma}

\begin{proof}
We first argue that $\dim \mathfrak K_{s,y}=d_s$ for all $y\in U$. By \cite[Lemma 5.5]{GiMe01} the dimension of $\mathfrak K_{s,y_0}$ is $d_s=\sum_j L_{s,j}$. This is also the number of zeros of $q_s(y_0,\cdot)$ in $D_{s,\eps}$ counting multiplicity (i.e., the order of vanishing of $q_s(y_0,\cdot)$ at $\sigma_s$, its single zero). To see this, note that the elements $(\sigma-\sigma_s)^{L_j} \phi^{K_s}_{j,0}(y_0,\sigma)$ are holomorphic at $\sigma=\sigma_s$ and that their values there, therefore also nearby, form a basis of $K_s$. We noted that the $\beta_{s,j}(\sigma)$ in \eqref{BasisOfRPerp} give a basis of $R_s^\perp$ when $\sigma=\sigma_s$, hence also for $\sigma$ near $\sigma_s$. With respect to these bases, the matrix of $\PP_s(y_0,\sigma)$ is diagonal with entries $(\sigma-\sigma_s)^{L^j}$, therefore $\det\PP_s(y_0,\sigma)=\prod_j(\sigma-\sigma_s)^{L_j}$ modulo a nonvanishing factor. If $y\in U$ is arbitrary and $\sigma'\in D_{s,\eps}\cap \sing_b(\P_y)$, then the same argument gives that
\begin{equation*}
\dim \set{\phi\in \mathfrak K_{s,y}: \pole(\phi)=\set{\sigma'}}
\end{equation*}
is equal to the order of vanishing of $q_s(y,\cdot)$ at $\sigma'$. Since $\mathfrak K_{s,y}$ is the direct sum of these spaces, the dimension of $\mathfrak K_{s,y}$ is the number of zeros, $d_s$ again, of $q_s(y,\cdot)$ counting multiplicity. Thus $\dim \mathfrak K_{s,y}=d_s$.

We now show that there is a neighborhood $U'$ of $y_0$ in $U$ such that the $d_s$ functions $\phi^{K_s}_{j,\ell}(y,\cdot)$ are linearly independent for every $y\in U'$. If not, there is a sequence $\set{y_k}_{k=1}^\infty\subset U$ converging to $y_0$ and for each $k$, numbers $a^{j,\ell}_k$ not all zero, such that 
\begin{equation*}
\sum_{j,\ell} a^{j,\ell}_k \phi^{K_s}_{j,\ell}(y_k,\cdot)=0.
\end{equation*}
We may assume $\sum_{j,k}|a^{j,\ell}_k|^2=1$ and then, passing to a subsequence, that the sequence $\set{a^{j,\ell}_k}_{k=1}^\infty$ converges, say $\lim_{k\to\infty}a^{j,\ell}_k=a^{j,\ell}$. The functions $\phi^{K_s}_{j,\ell}$ are in particular defined and continuous when $\eps/2<|\sigma-\sigma_s|<\eps$, so 
\begin{equation*}
0=\lim_{k\to \infty}\sum_{j,\ell} a^{j,\ell}_k \phi^{K_s}_{j,\ell}(y_k,\cdot)=\sum_{j,\ell} a^{j,\ell} \phi^{K_s}_{j,\ell}(y_0,\cdot)
\end{equation*}
if $|\sigma-\sigma_s|>\eps/2$ (see Remark \ref{SmoothnessOfFrame}). Since the $\phi^{K_s}_{j,\ell}(y_0,\sigma)$ are meromorphic in $\C$, the equality holds everywhere. Since not all $a_{j,\ell}$ are zero, we reach the conclusion that the elements \eqref{aBasis} are linearly dependent, a contradiction. Thus it must be that the $\phi^{K_s}_{j,\ell}(y,\cdot)$ are linearly independent for $y$ near $y_0$. This completes the proof of the lemma.
\end{proof}

Replacing $U$ by $U'$ allows us to assume:
\begin{enumerate}
\item [(5)] \label{IndependenceOnU} the elements $\phi^{K_s}_{j,\ell}$ form a basis of $\mathfrak K_{s,y}$ for each $y\in U$ and $s=1,\dotsc,S$.
\end{enumerate}
\begin{equation}
\display{310pt}{The elements $\phi^{K_s}_{j,\ell}$ form a basis of $\mathfrak K_{s,y}$ for each $y\in U$ and $s=1,\dotsc,S$.}
\end{equation}

\medskip
Suppose now that $\phi\in \BB^\infty(U;\kerb)$. Passing to trivializations, there is a smooth function 
\begin{equation*}
u:(U\times\Sigma)\minus \sing_e(\P)\to \Dom
\end{equation*}
such that $v_y= \P_y u_y$ is holomorphic in $\Sigma$ for each $y\in U$ and
\begin{equation*}
\phi=\ss_\Sigma(u)
\end{equation*}
The function $u_y$ is meromorphic in $\Sigma$ with poles in $\sing_b(\P_y)$. We will omit $y$ and $\sigma$ from the notation in the following few lines. Decomposing $u$ pointwise over $U\times D_{s,\eps}$ as
\begin{equation*}
u = u^{K_s}+u^{K_s^\perp}
\end{equation*}
according to $K_s\oplus K_s^\perp = \Dom$, similarly
\begin{equation*}
v = v^{R_s^\perp}+v^{R_s},
\end{equation*}
we obtain using \eqref{FrakA} that
\begin{equation*}
u^{K_s} = \PP_s^{-1}(v^{R_s^\perp}-\p^1_{s,2}(\p^2_{s,2})^{-1}v^{R_s}),\qquad u^{K_s^\perp}=(\p^2_{s,2})^{-1}(v^{R_s^\perp} - \p^2_{s,1}u^{K_s}).
\end{equation*}
Since $(\p^2_{s,2})^{-1}(v^{R_s^\perp})$ is smooth on $U\times D_{s,\eps}$ and holomorphic in the second variable, 
\begin{equation*}
u \equiv u^{K_s}-(\p^2_{s,2})^{-1}\p^2_{s,1}u^{K_s}
\end{equation*}
modulo a smooth function on $U\times D_{s,\eps}$ depending holomorphically on $\sigma$. Since $\ss_{D_{s,\eps}}u^{K_s}$ is an element of $\mathfrak K_{s,y}$ for each $y\in U$, Lemma~\ref{PointwiseBasisLemma} (i.e. Condition (5)) gives, for each such $y$, unique numbers $f^{j,\ell}_s(y)$ such that
\begin{equation*}
\ss_{D_{s,\eps}}(u^{K_s})=\sum_{j,\ell} f^{j,\ell}_s\phi^{K_s}_{j,\ell},
\end{equation*}
that is,
\begin{equation*}
u^{K_s}-\sum_{j,\ell} f^{j,\ell}_s\phi^{K_s}_{j,\ell}
\end{equation*}
is holomorphic in $\sigma$ for each $y\in U$ when $\sigma\in D_{s,\eps}$. So the same is true of
\begin{equation*}
u-\big(\sum_{j,\ell} f^{j,\ell}_s\phi^{K_s}_{j,\ell}-(\p^2_{s,2})^{-1}\p^2_{s,1}\sum_{j,\ell} f^{j,\ell}_s\phi^{K_s}_{j,\ell}\big),
\end{equation*}
hence passing to singular parts and adding we also have, with 
\begin{equation}\label{PointwiseBasis}
\phi =\sum_{s,j,\ell} f^{j,\ell}_s \big(\phi^{K_s}_{j,\ell}-\ss_{D_{s,\eps}}(\p^2_{s,2})^{-1}\p^2_{s,1} \phi^{K_s}_{j,\ell}\big),
\end{equation}
that $u_y-\phi_y$ is holomorphic in $\bigcup_{s} D_{s,\eps}$ for each $y\in U$ and $s=1,\dotsc,S$. Using that, as a consequence of \eqref{FiniteSpecb}, $\spec_b(\P_y)\subset \bigcup_s D_{s,\eps}$, we conclude that $u_y-\phi_y$ is holomorphic in $\Sigma$. It follows that the elements 
\begin{equation}\label{phihat}
\phi^s_{j,\ell}=\phi^{K_s}_{j,\ell}-\ss_{D_{s,\eps}}\big((\p^2_{s,2})^{-1}\p^2_{s,1}\phi^{K_s}_{j,\ell}\big)
\end{equation}
with $s=1,\dotsc,S$, $j=1,\dotsc,J_s$, $\ell=0,\dotsc,L_{s,j}-1$,  form a basis of $\kerb_y$ for each $y\in U$. By means of the trivialization of $\Ha_1$ over $U$ we view $\phi^s_{j,\ell}$ as an element of $\BB^\infty(U;\kerb)$ with singularities within $U\times D_{s,\eps/2}$. 

\section{Smoothness of transition functions}\label{sec-TransitionFunctions}

The proof of the following theorem will complete our proof of Theorem~\ref{TheKernelBundle}.

\begin{theorem}\label{TheFrame}
The $\phi^s_{j,\ell}$ in \eqref{phihat} are elements of $\BB^\infty(U;\kerb)$ giving a basis of $\kerb_y$ for each $y\in U$. If $\phi\in \BB^\infty(U;\kerb)$ then 
\begin{equation}\label{PhiAsLC}
\phi=\sum_{s,j,\ell} f^{j,\ell}_s\phi^s_{j,\ell}\quad \text{for }y\in U
\end{equation}
with smooth functions $f^{j,\ell}_s:U\to\C$.
\end{theorem}

\begin{proof}
We have already seen that, as a consequence of Condition (5) in the previous section and \eqref{PointwiseBasis}, if $\phi\in \BB^\infty(U;\kerb)$ then there are unique functions $f^{j,\ell}_s:U\to \C$ such that \eqref{PhiAsLC} holds. The task now is to show that the $f^{j,\ell}_s$ are smooth. 

Define $\kerb^*$ to be the meromorphic kernel bundle of $\scrP^*$; this will not necessarily be the dual bundle of $\kerb$ but the notation is convenient. The following theorem is the key result.

\begin{theorem}\label{NonDegPairing}
Let $\Omega\Subset \Sigma$ be open with smooth positively oriented boundary and $\sing_b(\scrP_y)\subset \Omega$. 
 The sesquilinear pairing
\begin{equation}\label{GreenMellinForm}
\begin{gathered}
\kerb_y\times\kerb^*_y\to \C,\\
[\phi,\psi]_y^\flat=\frac{1}{2\pi}\oint_{\partial\Omega} \big(\phi(\sigma),\scrP_y^\star (\overline \sigma) \psi(\overline \sigma)\big)_{\Ha_{1,y}}d\sigma
\end{gathered}
\end{equation}
is nondegenerate.
\end{theorem}

The theorem is proved below. Assuming its validity, we proceed as follows. Reverting to trivializations, write $\P^*(\sigma)$ (with $\sigma\in \overline\Sigma$) for the local version of $\scrP^*(\sigma)$ over $U$. Possibly after shrinking $U$ about $y_0$ we can carry out the constructions we did for $\P(\sigma)$ in obtaining the elements $\phi^s_{j,\ell}$ to obtain elements 
\begin{equation*}
\psi^{j,\ell}_s\in\BB^\infty(U;\kerb^*)
\end{equation*}
giving a pointwise basis of $\kerb^*$ over $U$. In doing this we note that $\sing_b(\scrP^*_y)$ is the conjugate set, $\overline{\sing_b(\scrP_y)}$, of $\sing_b(\scrP_y)$ and take advantage notationally of the fact that by \cite[Lemma 6.2]{GiMe01} we can use the same indices $s,j,\ell$ as for the $\phi^s_{j,\ell}$. We also arrange that the singularities of $\psi^{j,\ell}_s$ lie within $U\times \overline D_{s,\eps/2}$, the ``conjugate'' of $U\times D_{s,\eps/2}$.

If $\phi\in \B^\infty(U,\kerb)$ and $\psi\in \B^\infty(U,\kerb^*)$, then
\begin{equation*}
U\ni y\mapsto [\phi(y),\psi(y)]_y^\flat=\frac{1}{2\pi}\sum_s\oint_{\partial D_{s,\eps}} \big(\phi(y,\sigma),\P_y^\star (\overline \sigma) \psi(y,\overline \sigma)\big)_{H_1}d\sigma\in \C
\end{equation*}
is smooth, simply because the integrand is smooth in the complement in $U\times\Sigma$ of $\sing_e(\scrP)\cap (U\times \Sigma)$. Consequently, the functions defined by 
\begin{equation*}
a^{s,j',\ell'}_{s',j,\ell}(y)=[\phi^s_{j,\ell}(y),\psi^{j',\ell'}_{s'}(y)]_y^\flat
\end{equation*}
are smooth, as are the functions
\begin{equation*}
b^{j',\ell'}_{s'}(y)=[\phi(y),\psi^{j',\ell'}_{s'}(y)]_y^\flat. 
\end{equation*}
Since the $\phi^s_{j,\ell}(y)$ and $\psi^{j',\ell'}_{s'}(y)$ form bases of, respectively $\kerb_y$ and $\kerb^*_y$, the nondegeneracy of \eqref{GreenMellinForm} implies that the matrix $a=[a^{s,j',\ell'}_{s',j,\ell}]$ is nonsingular. If $\phi\in \BB^\infty(U;\kerb)$ then \eqref{PhiAsLC} gives
\begin{equation*}
b^{j',\ell'}_{s'}=[\phi,\psi^{j',\ell'}_{s'}]^\flat=\sum_{s,j,\ell}f^{j,\ell}_s[\phi^s_{j,\ell},\psi^{j',\ell'}_{s'}]^\flat = \sum_{s,j,\ell}f^{j,\ell}_sa^{s,j',\ell'}_{s',j,\ell}.
\end{equation*}
Since $a$ is nonsingular and its entries and those of $b=[b^{j',\ell'}_{s'}]$ are smooth, so are the $f^{j,\ell}_s$. This completes the proof of Theorem~\ref{TheFrame}.
\end{proof}

\begin{proof}[Proof of Theorem \ref{NonDegPairing}]
The assertion is a pointwise statement, so we may assume $y=y_0$ throughout the proof to take advantage of the notation introduced so far; we also work with the trivializations near $y_0$ of the various Hilbert space bundles. The proof we give here collects arguments spread throughout Sections 5, 6, and 7 of \cite{GiMe01}, used there to prove an analogous statement. Our notation here is slightly different from the one used there.

If $\phi\in \kerb_{y_0}$, then $\phi=\sum_{s}\phi^s$ with $\phi^s=\ss_{D_{s,\eps}}\phi$. The element $\phi^s$ has a pole only at $\sigma_s$ and is of the form
\begin{equation*}
\phi^s=\phi^{K_s}-\ss_{D_{s,\eps}}\big((\p^2_{s,2})^{-1}\p^2_{s,1}\phi^{K_s}\big)
\end{equation*}
with a unique $\phi^{K_s}\in \mathfrak K_{s,y_0}$, that is,
\begin{equation*}
\ss_{D_{s,\eps}}\PP_s\phi^{K_s}=0\quad \text{ and }\quad \ss_{D_{s,\eps}}\phi^{K_s}=\phi^{K_s}.
\end{equation*}

For each $\sigma_s\in \sing_b(\scrP_{y_0})$, the operator $\P_{y_0}^*(\sigma)=\P_{y_0}(\overline\sigma)^*$ decomposes near $\overline \sigma_s$ in the same fashion as $\P_y(\sigma)$ in \eqref{FrakA}, namely
\begin{equation*}
\begin{bmatrix}
\q^1_{s,1} & \q^1_{s,2}\\
\q^2_{s,1} & \q^2_{s,2}
\end{bmatrix}
:
\begin{matrix}
R_s^\perp\\\oplus\\ R_s\cap \Dom^*
\end{matrix}
\to
\begin{matrix}
K_s\\\oplus\\ \dot K_s^\perp 
\end{matrix}
\end{equation*}
where $\dot K_s^\perp$ is the subspace of $H_1$ orthogonal to $K_s$ (so $K^\perp_s=\dot K_s^\perp\cap \Dom$). Let 
\begin{equation}\label{FiniteQ}
\QQ_s=\q^1_{s,1}- \q^1_{s,2}(\q^2_{s,2})^{-1}\q^2_{s,1}
\end{equation}
and note in passing that $\QQ_s(\overline \sigma)^*=\PP_s(\sigma)$. An arbitrary element of $\kerb_{y_0}^*$ decomposes as $\psi=\sum\psi_s$ where $\psi_s=\ss_{\overline D_{s,\eps}}\psi$ has pole only at $\overline \sigma_s$ and has the form
\begin{equation*}
\psi_s=\psi_{R_s^\perp}-\ss_{\overline D_{s,\eps}}(\q^2_{s,2})^{-1}\q^2_{s,1}\psi_{R^\perp_{s}},
\end{equation*}
again with a unique
\begin{equation*}
\psi_{R_s^\perp} \in \mathfrak K^*_{s,y_0} = \set{\ss_{\overline D_{s,\eps}}(\psi):\psi\in \Mero(\overline D_{s,\eps},R^\perp_s),\ \QQ_s(y_0,\cdot)\psi\in \Hol(\overline D_{s,\eps},K_s)}.
\end{equation*}

We will first show that with $\phi\in \kerb_{y_0}$ and $\psi\in \kerb^*_{y_0}$ decomposed as indicated,
\begin{equation}\label{Reduced}
[\phi,\psi]_y^\flat=\frac{1}{2\pi}\sum_s\oint_{\gamma_s} \big(\phi^{K_s}(\sigma),\QQ_s(\overline \sigma) \psi_{R^\perp_s}(\overline \sigma)\big)_{H_1}d\sigma
\end{equation}
where $\gamma_s=\partial D_{s,\eps}$ with counterclockwise orientation.

The integrand in \eqref{GreenMellinForm} at $y_0$ is meromorphic in $\Sigma$ with poles at the various $\sigma_s$, so equal, after trivializations, to
\begin{equation*}
\frac{1}{2\pi}\sum_{s}\oint_{\gamma_s} \big(\phi_s(\sigma),\P_{y_0}^\star (\overline \sigma) \psi(\overline \sigma)\big)_{H_1}d\sigma.
\end{equation*}
We have
\begin{multline}\label{UnClean}
\oint_{\gamma_s} (\phi_s(\sigma),\P_{y_0}^\star (\overline \sigma) \psi(\overline \sigma))_{H_1}d\sigma\\
=\oint_{\gamma_s} \big(\phi^{K_s}(\sigma)-(\p^2_{s,2})^{-1}\p^2_{s,1}\phi^{K_s}(\sigma),\P_{y_0}^\star (\overline \sigma) \psi(\overline \sigma)\big)_{H_1}d\sigma
\end{multline}
because the difference of the integrands in the left and the right hand sides is holomorphic near $D_{s,\eps}$. The right hand side of this identity is in turn equal to
\begin{multline*}
\oint_{\gamma_s} \big(\P_{y_0}(\sigma)\big(\phi^{K_s}(\sigma)-(\p^2_{s,2})^{-1}\p^2_{s,1}\phi^{K_s}(\sigma)\big),\psi(\overline \sigma)\big)_{H_2}d\sigma\\
=\oint_{\gamma_s} \big(\P_{y_0}(\sigma)\big(\phi^{K_s}(\sigma)-(\p^2_{s,2})^{-1}\p^2_{s,1}\phi^{K_s}(\sigma)\big),\psi_s(\overline \sigma)\big)_{H_2}d\sigma.
\end{multline*}
The argument justifying the equality in \eqref{UnClean} now gives that the last integral is equal to
\begin{equation*}
\oint_{\gamma_s} \big(\P_{y_0}(\sigma)\big(\phi^{K_s}(\sigma)-(\p^2_{s,2})^{-1}\p^2_{s,1}\phi^{K_s}(\sigma)\big),\psi_{R_{s}^\perp}(\overline \sigma)-(\q^2_{s,2})^{-1}\q^2_{s,1}\psi_{R^\perp_{s}}(\overline \sigma)\big)_{H_2}d\sigma.
\end{equation*}
Consequently,
\begin{multline*}
\oint_{\gamma_s} \big(\phi(\sigma),\P_{y_0}^\star (\overline \sigma) \psi(\overline \sigma)\big)_{H_1}d\sigma\\
=\oint_{\gamma_s} \big(\phi^{K_s}(\sigma)-(\p^2_{s,2})^{-1}\p^2_{s,1}\phi^{K_s}(\sigma),\P_{y_0}^\star (\overline \sigma) (\psi_{R_{s}^\perp}(\overline \sigma)-(\q^2_{s,2})^{-1}\q^2_{s,1}\psi_{R^\perp_{s}}(\overline \sigma))\big)_{H_1}d\sigma.
\end{multline*}
In view of \eqref{FiniteQ},
\begin{align*}
\P_{y_0}^\star(\psi_{R_{s}^\perp}-(\q^2_{s,2})^{-1}\q^2_{s,1}\psi_{R^\perp_{s}})
&=[\q^1_{s,1}\psi_{R_{s}^\perp}-\q^1_{s,2}(\q^2_{s,2})^{-1}\q^2_{s,1}\psi_{R^\perp_{s}}]\\
&\qquad +[\q^2_{s,1}\psi_{R_{s}^\perp}-\q^2_{s,2}(\q^2_{s,2})^{-1}\q^2_{s,1}\psi_{R^\perp_{s}}]\\
&=\QQ_s\psi_{R^\perp_{s}},
\end{align*}
hence
\begin{equation*}
\oint_{\gamma_s} \big(\phi(\sigma),\P_{y_0}^\star (\overline \sigma) \psi(\overline \sigma)\big)_{H_1}d\sigma
=\sum_s \oint_{\gamma_s} \big(\phi^{K_s}(\sigma),\QQ_s(\overline \sigma) \psi_{R^\perp_s}(\overline \sigma)\big)_{H_1}d\sigma
\end{equation*}
which gives \eqref{Reduced}. Note that
\begin{equation*}
\oint_{\gamma_s} \big(\phi(\sigma),\QQ_s(\overline \sigma) \psi(\overline \sigma)\big)_{H_1}d\sigma
=\oint_{\gamma_s} \big(\PP_s(\sigma) \phi^{K_s}(\sigma), \psi_{R^\perp_s}(\overline \sigma)\big)_{H_2}d\sigma.
\end{equation*}

We now show that for each $s$, the pairing
\begin{equation*}
\begin{gathered}
\mathfrak K_{s, y_0}\times\mathfrak K^*_{s, y_0}\to \C,\\
[\phi^{K_s},\psi_{R^\perp_s}]_{s,y_0}^\flat=\frac{1}{2\pi}\oint_{\gamma_s} \big(\phi^{K_s}(\sigma),\QQ_s (\overline \sigma) \psi_{R^\perp_s}(\overline \sigma)\big)_{H_1}d\sigma
\end{gathered}
\end{equation*}
is nondegenerate; this is Theorem 6.4 in \cite{GiMe01}. By Lemma 6.2 of that paper, there are $\psi^{j,0}_{R^\perp_s}(y_0,\sigma)$ meromorphic in $\C$ with pole only at $\overline \sigma_s$, of order $L_{s,j}$, such that
\begin{enumerate}
\item [(a${}'$)] $\psi^{j,\ell}_{R^\perp_s}(y_0,\sigma)=\ss_{\overline D_{s,\eps}}\big((\sigma-\overline\sigma_s)^\ell \psi^{j,0}_{R^\perp_s}(y_0,\sigma)\big)$, $j=1,\dotsc,J_s$, $\ell=0,\dotsc,L_{s,j}-1$ is a basis of $\mathfrak K^*_{s,y_0}$,
\item [(b${}'$)] the values of the $(\sigma-\overline \sigma_s)^{L_{s,j}}\psi^{j,0}_{R^\perp_s}(y_0,\sigma)$ at $\overline \sigma_s$ from a basis of $R^\perp_s$,
\item [(c${}'$)] the values of
\begin{equation}\label{BasisOfK}
\alpha^{s,j}(\sigma)=\QQ_s(y_0,\sigma)\psi^{j,0}_{R^\perp_s}(y_0,\sigma)
\end{equation}
at $\overline \sigma_s$ form a basis of $K_s$, and  
\item [(d${}'$)] each of the functions
\begin{equation*}
\sigma\mapsto \big((\sigma-\sigma_s)^{L_{s,j'}} \phi^{K_s}_{j,0}(y_0,\sigma),\QQ_s(y_0,\sigma)\psi^{j',0}_{R^\perp_s}(\overline \sigma)\big)_{H_2},\quad j,j'=1,\dotsc,J_s
\end{equation*}
is holomorphic in a neighborhood of the closure of $D_{s,\eps}$ with value $\delta_{jj'}$ at $\sigma=\sigma_s$.
\end{enumerate}
Properties (a${}'$)--(c${}'$) are the analogues for $\QQ_s$ of the properties of the $\phi^{K_s}_{j,\ell}$. The fourth property links the two bases in a convenient manner.

Suppose 
\begin{equation*}
\phi^{K_s}=\sum f^{j,\ell}\phi^{K_s}_{j,\ell}, \quad \psi_{R^\perp_s}=\sum g_{j,\ell}\psi^{j,\ell}_{R^\perp_s}.
\end{equation*}
Then 
\begin{equation*}
[\phi^{K_s},\psi_{R^\perp_s}]_{s,y_0}^\flat =\frac{1}{2\pi}\sum_{j,\ell,j',\ell'} f^{j,\ell}\overline g_{j'\ell'}\oint_{\gamma_s}\big(\phi^{K_s}_{j,\ell}(\sigma),\QQ_s (\overline \sigma) \psi^{j',\ell'}_{R^\perp_s}(\overline \sigma)\big)_{H_1}d\sigma
\end{equation*}
The difference
\begin{equation*}
\big(\phi^{K_s}_{j,\ell}(\sigma),\QQ_s (\overline \sigma)\psi^{j',\ell'}_{R^\perp_s}(\overline \sigma)\big)_{H_1} - \big((\sigma-\sigma_0)^\ell\phi^{K_s}_{j,0},\QQ_s (\overline \sigma)((\overline \sigma-\overline \sigma_0)^{\ell'}\psi^{j',0}_{R^\perp_s}(\overline \sigma))\big)_{H_1}
\end{equation*}
is holomorphic in a neighborhood of $D_{s,\eps}$, so
\begin{align*}
2&\pi[\phi^{K_s},\psi_{R^\perp_s}]_{s,y_0}^\flat \\
&=\sum_{j,\ell,j',\ell'} f^{j,\ell}\overline g_{j'\ell'}\oint_{\gamma_s}(\sigma-\sigma_0)^{\ell+\ell'} \big(\phi^{K_s}_{j,0}(\sigma),\QQ_s (\overline \sigma)(\psi^{j',0}_{R^\perp_s}(\overline \sigma))\big)_{H_1}d\sigma\\
&=\sum_{j,\ell,j',\ell'} f^{j,\ell}\overline g_{j'\ell'}\oint_{\gamma_s}(\sigma-\sigma_0)^{\ell+\ell'-L_{s,j}} \big((\sigma-\sigma_0)^{L_{s,j}}\phi^{K_s}_{j,0}(\sigma),\QQ_s (\overline \sigma)(\psi^{j',0}_{R^\perp_s}(\overline \sigma))\big)_{H_1}d\sigma
\end{align*}
By virtue of (d${}'$) the last integral vanishes unless $j=j'$ and $\ell+\ell'=L_j-1$ in which case the value is $2\pi\im$. Thus  we arrive at
\begin{equation*}
[\phi^{K_s},\psi_{R^\perp_s}]_{s,y_0}^\flat = 
\im \sum_{j=1}^{J_s}\sum_{\ell=0}^{L_j-1} f^{j,\ell}\overline g_{j,L_j-\ell-1}.
\end{equation*}
It follows immediately that the condition $[\phi^{K_s},\psi_{R^\perp_s}]_{s,y_0}^\flat=0$ for all $\psi^{R^\perp_s}\in \mathfrak K_{s,y_0}$ implies $\phi^{K_s}=0$. 
This completes the proof of Theorem \ref{NonDegPairing}.
\end{proof}

\section{The trace bundle of a wedge operator}\label{sec-TraceWedge}

Let $\M$ be a manifold whose boundary $\N=\partial M$ fibers over a connected manifold $\Y$ with typical fiber $\Z$, assumed to be compact:
\begin{equation}\label{BdyFiberBundle}
\display{120pt}{\begin{center}
\begin{picture}(120,50)
\put(35,44){$\Z\embed\N$	
\put(-9,-20){$\bigg\downarrow$}
\put(-1,-18){\small $\wp$}
\put(-10,-40){$\Y$}
}
\end{picture}
\end{center}}
\end{equation}
Let $E$, $F\to \M$ be vector bundles. Recall that a wedge differential operator on $\M$ is an operator $A\in x^{-m}\Diff^m_e(\M;E,F)$ where $\Diff^m_e(\M;E,F)$ is the class of edge-differential operators with boundary structure associated to \eqref{BdyFiberBundle} defined by Mazzeo \cite{Maz91}; here and throughout the rest of this section, $x$ denotes a defining function of $\N$ which is positive in $\open\M$. The notion of $w$-ellipticity of $A$, invariantly defined in \cite[Section 2]{GKM4}, is equivalent to ellipticity of $x^mA$ in the sense of Mazzeo (op.~cit., p.~1620). Examples of these operators include regular elliptic operators on a manifold with boundary (where $\Z$ reduces to a point) and elliptic cone operators, where $\Y$ is discrete. We will assume in this section that $A$ is $w$-elliptic.

Let $\pi:\N^\wedge\to\N$ be the closure of the inward pointing normal bundle of $\N$, an $\overline\R_+$ bundle, $E^\wedge=\pi^*E$, likewise $F^\wedge$, let $\Z_y=\wp^{-1}(y)$, $\Z^\wedge_y=\pi^{-1}\Z_y$, $E_{\Z_y}=E|_{\Z_y}$, similarly $E^\wedge_{\Z^\wedge_y}$. 

The normal family of $A\in x^{-m}\Diff^m_e(\M;E,F)$, see \cite[Definition 2.13]{GKM4}, is an invariantly defined family of cone operators 
\begin{equation*}
T^*\Y\ni \pmb\eta\mapsto A_\wedge(\pmb \eta)\in x^{-m}\Diff^m_b(\Z^\wedge_y;E^\wedge_{\Z_y},F^\wedge_{\Z_y}).
\end{equation*}
This family along the zero section of $T^*\Y$ gives a family $y\mapsto \bA_y$ of cone-differential operators.

Fix $\gamma\in \R$ and let $\trb_y$ denote the set of sections of $E^\wedge_{\Z_y}$ which are finite sums of the form
\begin{equation}\label{ProtoTrace}
u=\sum_{\substack{\sigma\in \spec_b(\bP_y)\\ \gamma-m<\Im\sigma<\gamma}}\sum_{\ell=0}^{L_\sigma} a_{\sigma,\ell}x_\wedge^{\im \sigma}\log^\ell x_\wedge
\end{equation}
that satisfy $\bA_y u=0$. Here $\bP_y=x^{m}\bA_y$, an elliptic $b$-operator for each $y\in \Y$. The function $x_\wedge:\N^\wedge\to \R$ is smooth and linear in the fibers, positive off the zero section (take $x_\wedge=dx$ for instance) and the $a_{\sigma,\ell}$ are smooth sections of $E$ along $\Z_y=\wp^{-1}(y)$. Let
\begin{equation*}
\trb=\bigsqcup_{y\in\Y}\trb_y
\end{equation*}
and let $\pi_\Y:\trb\to\Y$ be the canonical map. 

Define
\begin{equation*}
\spec_e(A)=\set{(y,\sigma)\in \Y\times \C:\sigma\in \spec_b(\bP_y)}.
\end{equation*}

\begin{theorem}\label{TheTraceBundle}
Suppose that $A\in x^{-m}\Diff^m_e(\M;E,F)$ is $w$-elliptic and 
\begin{equation}\label{b-specGap}
\spec_e(A)\cap \big(\Y\times\set{\sigma\in \C:\Im\sigma= \gamma, \gamma-m}\big)=\emptyset.
\end{equation}
Then $\pi_\Y:\trb\to\Y$ is a smooth vector bundle over $\Y$, the trace vector bundle of $A$. The $C^\infty$ sections of $\trb$ are the sections pointwise of the form \eqref{ProtoTrace} which are smooth as sections of $E^\wedge$ over $\open \N^\wedge$. 
\end{theorem}

This is an application of Theorem~\ref{TheKernelBundle}. Let $\scrP_y$ be the indicial family of $\bP_y$, namely $\scrP_y(\sigma)$ is the restriction of $x^{-\im\sigma}P_y x^{\im \sigma}$ to the space of distributional sections of $E$ along $\Z_y$. Let $\Ha_{1,y}=L^2(\Z_y;E_{\Z_y})$, $\Ha_{2,y}=L^2(\Z_y;F_{\Z_y})$, and $\scrD_y$ is the Sobolev space $H^m(\Z_y;E_{\Z_y})$. These are the fibers of Hilbert space bundles $\Ha_1$, $\Ha_2\to\Y$, see \cite[Section 3]{GKM4} and a set-theoretical subbundle $\scrD\subset \Ha_1$. Because the trivializations of $\Ha_1$ and $\Ha_2$ are constructed using diffeomorphisms trivializing $\N\to\Y$, $\scrD$ is a bundle with fiber $H^m(\Z,E_\Z)$ where $E_\Z$ is bundle-isomorphic to the restriction of $E$ to some fiber of $\wp:\N\to\Y$. The properties required in Section~\ref{sec-SetUp} hold here for $\scrP$ on $\Y\times\Sigma$, where
\begin{equation*}
\Sigma=\set{\sigma\in \C:\gamma-m<\Im\sigma<\gamma}. 
\end{equation*}
Since $\sing_b(\scrP_y)=\spec_b(\bP_y)\cap\Sigma$, 
\begin{equation*}
\spec_e(A)\cap\big(\Y\times\Sigma\big)=\sing_e(\scrP).
\end{equation*}
The hypothesis \eqref{b-specGap} and the ellipticity of $A$ imply that \eqref{FiniteSpecb} holds, so we have a well defined smooth vector bundle associated with $\scrP$. 

We now define a bundle isomorphism $\kerb\to\trb$. Given $y\in \Y$ and $\phi\in \kerb_y$ let 
\begin{equation*}
S_y(\phi) = \frac{1}{2\pi} \oint_\Gamma x_\wedge^{\im \sigma}\phi(\sigma)\, d\sigma. 
\end{equation*}
Here $\Gamma$ is a simple closed smooth curve enclosing $\sing_b(\P_y)\cap \Sigma$ and the integral is computed with the counterclockwise orientation. Then $S_y(\phi)$ has the form \eqref{ProtoTrace} and since
\begin{equation*}
\bP_y S_y(\phi)=\frac{1}{2\pi} \oint_\Gamma x^{\im \sigma}\scrP_y(\sigma) \phi(\sigma)\, d\sigma
\end{equation*}
vanishes because the integrand is entire, $S_y(\phi)\in \trb_y$. If $u\in \trb_y$ and $\omega$ is a cut-off function equal to $1$ near $x_\wedge=0$, then the singular part of the Mellin transform of $\omega u$,
\begin{equation*}
\ss_\C\int_0^\infty x_\wedge^{-\im \sigma}\omega(x_\wedge) u(x_\wedge) \,\frac{d x_\wedge}{x_\wedge},
\end{equation*}
is an element of $\kerb_y$ such that $S_y(\phi)=u$. So $S$ is a fiberwise isomorphism which we may use to give $\trb$ the structure of a $C^\infty$ vector bundle. It is easy to see that $S$ maps $\BB^\infty(\Y;\kerb)$ into elements pointwise of the form \eqref{ProtoTrace} that are smooth as sections of $E^\wedge$ over $\open \N^\wedge$, so the smooth structure of $\trb$ induced by that of $\kerb$ is the same as one would obtain from declaring as smooth the aforementioned sections.

\section{Example}\label{sec-Example}

The following example is a toy model of the situation that arises in the analysis of a boundary value problem for a symmetric second order elliptic operator near the boundary, $\Y$, (codimension $2$) of a codimension $1$ smooth submanifold $\Crack$ of the ambient manifold, with Dirichlet condition along $\Crack$.

Let $\Y$ be a compact connected manifold and let $E\to\Y$ be a Hermitian complex vector bundle. Further, let $a\in C^\infty(\Y;\End(E))$ be a self-adjoint section. Let $\wp:\Ha\to\Y$ be the Hilbert space bundle whose fiber over $y\in \Y$ is the space of $E_y$-valued $L^2$ functions on $[0,\pi]$:
\begin{equation*}
\Ha_y=L^2([0,\pi];E_y),
\end{equation*}
defined using the Lebesgue measure on $[0,\pi]$ and the Hermitian product on $E_y$. The unitary trivializations of $\Ha\to\Y$ are obtained from those of $E\to\Y$. Let 
\begin{equation*}
\scrD_y=\set{\phi\in L^2([0,\pi];E_y):\phi \in H^2([0,\pi];E_y),\ \phi(0)=\phi(\pi)=0}.
\end{equation*}
Here $H^2([0,\pi];E_y)$ is the Sobolev space of order $2$. Define
\begin{equation*}
\qquad \scrD=\bigsqcup_{y\in \Y} \scrD_y,
\end{equation*}
a set-theoretical subbundle of $\Ha$. 

Define
\begin{equation*}
\scrP(\sigma):\scrD\subset\Ha\to\Ha
\end{equation*}
as
\begin{equation*}
\scrP_y(\sigma)=D_s^2+a+\sigma^2.
\end{equation*}
Thus $\scrP_y(\sigma)$ is closed, densely defined and Fredholm of index zero. The graph norms of the $\scrP_y(\sigma)$ are all equivalent to each other. Further, the adjoint family $\scrP^*$ has the same properties. So the results of the Sections~\ref{sec-Frames} and \ref{sec-TransitionFunctions} are applicable relative to open sets $\Sigma\subset \C$ satisfying \eqref{FiniteSpecb}.

To see that such open sets exist we describe $\sing_b(\scrP_y)$ explicitly in terms of the eigenvalues of $a(y)$. Decomposing $E_y$ as a direct sum of eigenspaces of $a(y)$ the problem becomes a family of scalar problems 
\begin{equation*}
(D_s^2+\mu_j+\sigma^2)\phi=0,\quad j=1,\dotsc,J_y
\end{equation*}
where the $\mu_j$ are the eigenvalues of $a(y)$. This equation has a nonzero solution with $\phi(0)=\phi(\pi)=0$ if and only if $-\mu_j-\sigma^2=k^2$ with $k\in \NN$. Thus
\begin{equation*}
\sing_b(\scrP_y)=\set{\pm\im \sqrt{\mu+k^2}:\mu\in\spec(a(y)),\ k\in \NN}.
\end{equation*}

Since $\Y$ is compact, the norm of $a$ is bounded, say $\tr(a^*a)<r^2$. Then the eigenvalues of $a(y)$ lie in the interval $(-r,r)$ and consequently, the numbers $\mu+k^2$, $\mu\in \spec(a(y))$, lie in the interval $I_k=(k^2-r,k^2+r)$. This interval and the interval $I_{k+1}$ are disjoint if and only if $k>(2r-1)/2$. This condition on $k\in \NN$ ensures that no point of $\sing_e(\scrP)$ belongs to the set $\Y\times \partial\Sigma$ where
\begin{equation*}
\Sigma=\set{\sigma\in \C:|\Im\sigma|<\sqrt{k^2+k+1/2}}
\end{equation*}
since $k^2+k+1/2$ is the midpoint between $k^2+r$ and $(k+1)^2-r$.


\end{document}